\documentclass[11pt]{amsart}

\usepackage{verbatim}

\usepackage{amsmath,amsthm,amssymb}
\usepackage[all]{xy}
\usepackage{xspace}
\usepackage{enumerate}

\newcommand{\Z}{\ensuremath{\mathbb{Z}}\xspace}

\newcommand{\C}{\ensuremath{\mathbb{C}}\xspace}

\newcommand{\nt}{\ensuremath{\mathbb{N}}\xspace}

\newcommand{\norm}[1]{\ensuremath{\|#1\|}}

\DeclareMathOperator{\gen}{span}

\DeclareMathOperator{\supp}{supp}

\theoremstyle{plain}
\newtheorem{thm}{Theorem}[section]
\newtheorem{prop}[thm]{Proposition}
\newtheorem{lem}[thm]{Lemma}
\newtheorem{cor}[thm]{Corollary}          

\theoremstyle{definition}
\newtheorem{defn}[thm]{Definition}
\newtheorem{exmp}[thm]{Example}

\theoremstyle{remark} 
\newtheorem{rem}[thm]{Remark}

\numberwithin{equation}{section}

\begin{document}
\title[Dilations of interaction groups]{Dilations of interaction
  groups that extend actions of Ore semigroups}
\author{Fernando Abadie}
\address{Centro de Matem\'atica-FC
%entro de Matem\'atica-Facultad de Ciencias
, Universidad de la Rep\'ublica. 11\,400 Igu\'a
4225, Montevideo, URUGUAY.} 
\email{fabadie@cmat.edu.uy}
\thanks{Partially supported by CNPq, Brazil, Processo
    151654/2006-9}



\date{}

\begin{abstract}
We show that every interaction group extending an action of an Ore
semigroup by injective unital endomorphisms of a $C^*$-algebra, admits
a dilation to an action of the corresponding enveloping group on
another unital $C^*$-algebra, of which the former is a
$C^*$-subalgebra: the interaction group is obtained by composing the
action with a conditional expectation. The dilation is
essentially unique if a certain natural condition of minimality is
imposed. If the action is induced by covering maps on the spectrum,
then the expectation is faithful.    
\end{abstract}

\maketitle

%\tableofcontents
\section{Introduction and preliminaries}
\par The notions of interaction groups and their crossed products have
been introduced and studied by Exel in \cite{exends},
with the aim of dealing with irreversible dynamical systems. The
mentioned paper emerges as a culmination of previous work in the
subject appeared in \cite{exend}, \cite{exinteractions},
\cite{exfindex}, \cite{exvershik}. Related work may be found as well
in \cite{nl}, \cite{br}, \cite{danilo}, \cite{deaconu}. 
\par Recently, Exel and Renault studied in \cite{exren} a family of
interaction groups that extend actions of some semigroups on unital
commutative $C^*$-algebras. 
\par Suppose 
that $X$ is a compact Hausdorff space and $\theta:X\to X$ is a
covering map. For $A=C(X)$, let $\alpha:A\to A$ be
the dual map of  
$\theta$, i.e.: $\alpha(a)=a\circ\theta$, which is a unital injective
endomorphism of $A$. In case there exists a transfer operator
(\cite{exend}) for $\alpha$, that is, a positive linear map 
$\mathcal{L}:A\to A$ such that
$\mathcal{L}(\alpha(a)b)=a\mathcal{L}(b)$, $\forall 
a,b\in A$, then $V:\mathbb{Z}\to B(A)$ (here $B(A)$ is the algebra of
bounded operators from $A$ into itself) given by  
$V_n=\begin{cases}\alpha^n & n\geq 0\\ \mathcal{L}^{-n}&
  n<0\end{cases}$ is called an interaction group
(Definition~\ref{defn:pig}). 
This interaction group is clearly an
extension of the action $\bar{\alpha}:\mathbb{N}\times 
A\to A$ given by $(n,a)\mapsto\alpha^n(a)$. Conversely, it can be
shown that if $W:\mathbb{Z}\to B(A)$ is an interaction group that
extends $\bar{\alpha}$, then $W_{-1}$ is a transfer operator for
$\alpha$, and $W$ is retrieved from $\alpha$ and $W_{-1}$ from the
construction above. That is: interaction groups that extend 
$\bar{\alpha}$ are in a natural bijection with transfer operators for
$\alpha$. In the same way, interaction groups that extend the action
of an Ore semigroup correspond to semigroups of transfer operators
corresponding to the endomorphisms of the action. In the case of
actions on commutative algebras the work in \cite{exren} shows that
one can replace transfer operators by \textit{cocycles} (see
Definition~\ref{defn:cocycle}). We will show later that an interaction
group as the one above can be written as the composition of an action
$\beta$ with a conditional expectation $F$:~$V_n=F\beta_n$, $\forall
n\in \Z$, a decomposition that reflects the combination of the 
deterministic and probabilistic elements included in the concept of
interaction group.               
\par On the other hand, it seems that interaction groups are closely
related with partial actions. Propositions~\ref{prop:pa}
and~\ref{prop:env} below are instances of this relation. Moreover,
under certain conditions one may construct interaction groups from
actions of groups and conditional expectations, in a way that
resembles the construction 
of partial actions by the restriction of global ones. In fact, suppose 
that $A$ is a $C^*$-subalgebra of the unital $C^*$-algebra $B$,
$F:B\to B$ is a conditional expectation with range $A$, and 
$\beta:G\times  
B\to B$ is an action of a group $G$ on~$B$. Let $F_t:B\to B$ be given
by $F_t:=\beta_tF\beta_{t^{-1}}$. Then $F_t$ is a conditional
expectation onto $\beta_t(A)$, $\forall t\in G$. It is not hard to
prove that if $FF_t=F_tF$, $\forall t\in G$, then $V:G\to B(A)$ such
that $V_t(a)=F(\beta_t(a))$, $\forall a\in A, t\in G$ is an
interaction group (provided $F(\beta_t(1_A))=1_A$, $\forall t\in G$,
see~\ref{prop:1} below).
\par With the same spirit of the work done in \cite{fa}, although with
different methods, 
we show in the present paper that any interaction group that extends
an action of an Ore semigroup by unital injective endomorphisms (for
instance those studied in \cite{exren}) is of this form, that is,  
it can be obtained by composing an action with a conditional
expectation. The existence of the action is due to Laca's Theorem 
(see \cite{ml} and Theorem~\ref{thm:ml} below) on the dilation of
actions of Ore semigroups. The conditional expectation is constructed
as the limit of the directed system of transfer operators
corresponding to the endomorphisms of the Ore semigroup action. 
\bigskip
\par The structure of the present paper is the following. In the rest
of this section we study some relations between interaction groups and
partial actions and we introduce the notion of dilation of an
interaction group. In the next section we prove our main result, 
Theorem~\ref{thm:main}, and in the final one we see how this theorem
applies, with a refinement, to the interaction groups studied by Exel 
and Renault in~\cite{exren}.   

\subsection{Interaction groups}
\par We show here how to get interaction groups from suitable pairs of
actions and conditional expectations. Recall that a partial
representation of a group $G$ on a Banach space $A$ is a map
$V:G\to B(A)$, the Banach algebra of bounded linear operators on $A$,
such that $\begin{cases} 
         V_e=Id & \textrm{($e$ the unit of $G$)}\\  
         V_{s^{-1}}V_sV_t=V_{s^{-1}}V_{st} & \forall s,t\in G\\ 
         V_sV_tV_{t^{-1}}=V_{st}V_{t^{-1}} & \forall s,t\in G \\ 
        \end{cases}$  
\begin{defn}\label{defn:pig} 
An interaction group is a triple $(A,G,V)$ where $A$ is a
unital $C^*$-algebra, $G$ is a group, and $V$ is a map from $G$ into 
$B(A)$,  which satisfies: 
\begin{enumerate}
 \item $V_t$ is a positive unital map, $\forall t\in G$.
 \item $V$ is a partial representation. 
 \item $V_t(ab)=V_t(a)V_t(b)$ if either $a$ or $b$ belongs to
       $V_{t^{-1}}(A)$. 
\end{enumerate}
If the group $G$ is understood we will put just $(A,V)$ (or even $V$
if $A$ is understood as well) instead of $(A,G,V)$. A morphism
$(A,G,V)\stackrel{\psi}{\to} (A',G,V')$ is a unital homomorphism of
$C^*$-algebras $\psi:A\to A'$ such that $\psi V_t=V'_t\psi$, $\forall
t\in G$.  
\end{defn} 
\par It will be useful for our purposes to consider the following
couple of categories, $\mathcal{T}_G$ and $\mathcal{D}_G$ associated
to a group~$G$. The objects of $\mathcal{T_G}$ are triples
$T=(B,\beta,F)$, where $\beta$ is an action of the group $G$ on the
unital $C^*$-algebra~$B$, and $F:B\to B$ is a conditional expectation,
that is, a norm one idempotent whose range is a $C^*$-subalgebra of
$B$. Recall that a conditional expectation $F$ is a positive 
$F(B)$-bimodule map. If $T=(B,\beta,F)$, 
$T'=(B',\beta',F')\in\mathcal{T}_G$, by a morphism $\phi:T\to T'$ we
mean a unital homomorphism of $C^*$-algebras $\phi:B\to B'$ such that
$\phi F=F'\phi$ and $\phi\beta_t=\beta_t'\phi$, $\forall t\in G$. The
category $\mathcal{D}_G$ is the full subcategory of $\mathcal{T}_G$
whose objects $(B,\beta,F)$ satisfy the following two conditions:
a)~$F\beta_tF(1)=F(1)$, $\forall t\in G$, and b)~$F_rF_s=F_sF_r$,
$\forall r,s\in G$, where $F_r=\beta_rF\beta_{r^{-1}}$, $\forall r\in
G$. Note that $F(1)$ is the unit of $F(B)$, and that $F_r$ is a
conditional expectation with range $\beta_r(F(B))$.       
\begin{prop}\label{prop:1}
Let $T=(B,\beta,F)\in\mathcal{T}_G$, and $A:=F(B)$. If $FF_t=F_tF$,
$\forall t\in G$, then:
\begin{enumerate}
\item $F_rF_s=F_sF_r$, $\forall r,s\in G$, and $F_rF_s$ is a
      conditional expectation with range $\beta_r(A)\cap\beta_s(A)$.    
\item\label{V} If $V_t:=F\beta_t|_A$, $\forall t\in G$, then the map $V:G\to
      B(A)$ given by $t\mapsto V_t$ is a partial representation and
      satisfies condition 3. of \ref{defn:pig}. Moreover the range of
      $V_t$ is $A\cap \beta_t(A)$, 
      $\forall t\in G$.
\item If $V$ is the map defined in \ref{V}, then $V$ is an
      interaction group if and only if $F(\beta_t(1_A))=1_A$ for every
      $t\in G$. That is: $V$ is an interaction group if and only if
      $T\in\mathcal{D}_G$.       
\end{enumerate}
\end{prop}
\begin{proof}
We have
\[\beta_rF\beta_{r^{-1}}\beta_sF\beta_{s^{-1}}
=\beta_sF_{s^{-1}r}F\beta_{s^{-1}}
=\beta_sFF_{s^{-1}r}\beta_{s^{-1}}
=\beta_sF\beta_{s^{-1}}\beta_rF\beta_{r^{-1}}\beta_s\beta_{s^{-1}}.\]
That is, $F_rF_s=F_sF_r$, and therefore $F_rF_s$ is a conditional 
expectation with range $F_r(B)\cap F_s(B)$. On the other hand 
$F_r(B)=\beta_rF\beta_{r^{-1}}(B)=\beta_rF(B)=\beta_r(A)$. Hence
$F_rF_s(B)=\beta_r(A)\cap\beta_s(A)$.
\par As for \textit{2.}, $V$ is a partial representation:  
\begin{gather*}
V_{s^{-1}}V_sV_t=FF_{s^{-1}}F\beta_t=FF_{s^{-1}}\beta_t=V_{s^{-1}}V_{st},\\
V_sV_tV_{t^{-1}}=F\beta_{st}F_{t^{-1}}F\beta_{t^{-1}}|_A
%=V_{st}F\beta_{t^{-1}}F\beta_t\beta_{t^{-1}}|_A
=V_{st}F\beta_{t^{-1}}Id_A=V_{st}V_{t^{-1}}.
\end{gather*}
\par If $x\in A$ and $a=V_{t^{-1}}(x)$, $b\in B$, then: 
\begin{gather*}
V_t(ab)=V_t(V_{t^{-1}}(x)b)=
F(F_t(x)\beta_t(b))
=F(F_tF(x)\beta_t(b))\\=F(FF_t(x)\beta_t(b))=FF_t(x)F(\beta_t(b))
=V_tV_{t^{-1}}(x)V_t(b)=V_t(a)V_t(b).
\end{gather*}
 Since $V_t(ba)=V_t(a^*b^*)^*$,
we have shown that $V$ satisfies condition \textit{3.} of
\ref{defn:pig}. On the other hand, 
$V_t(A)=F\beta_t(A)=FF_t(\beta_t(A))=A\cap\beta_t(A)$, because $FF_t$ is a
conditional expectation with range $A\cap\beta_t(A)\subseteq\beta_t(A)$.  
\par Now, if $V$ is an interaction group, 
then $F\beta_t(1_A)=V_t(1_A)=1_A$, $\forall t\in G$. Conversely, if 
$F\beta_t(1_A)=1_A$ $\forall t\in G$, then $V_t$ is a positive unital
map. In addition $V_e=F\beta_e|_A=FId_A=Id_A$; hence $V$ is an
interaction group. 
\end{proof}
\subsection{The partial action of an interaction group}
\par We will see now that every interaction group has naturally
associated a partial action of the group on the same algebra. Recall
that a partial action of a discrete group $G$ on a set $X$
is a pair $(\{X_t\}_{t\in G},\{\gamma_t\}_{t\in G})$ where, for every
$t\in G$, $X_t$ is a subset of $X$, $\gamma_t:X_{t^{-1}}\to X_t$ is a 
bijection, and $\gamma_{st}$ extends $\gamma_s\gamma_t$, $\forall
s,t\in G$. It is also assumed that $\gamma_e=id_X$. When $X$ is a
$C^*$-algebra, it is usually supposed that $X_t$ is an ideal and that
$\gamma_t$ is an isomorphism of $C^*$-algebras. So we warn the reader
that for the partial actions we consider in this paper the sets $X_t$ 
will be unital $C^*$-subalgebras.  

\begin{prop}\label{prop:pa}
Suppose that $V:G\to B(A)$ is an interaction group. For $t\in
G$ let $A_t:=V_t(A)$, and $\gamma_t:A_{t^{-1}}\to A_t$ be such that
$\gamma_t(a)=V_t(a)$. Then:
\begin{enumerate}
\item Every $A_t$ is a unital $C^*$-subalgebra of $A$ (with the same 
      unit), and $\gamma_t$ is an isomorphism  
      between $A_{t^{-1}}$ and $A_t$, $\forall t\in G$. 
\item The map $E_t:A\to A$ given by $E_t:=V_tV_{t^{-1}}$ is a
      conditional expectation onto $A_t$, $\forall t\in G$, and
      $E_rE_s=E_sE_r$, $\forall r,s\in G$.  
\item The pair $\gamma:=(\{A_t\}_{t\in G},\{\gamma_t\}_{t\in G})$ is a
      partial action of $G$ on~$A$.   
\end{enumerate}
\end{prop}
\begin{proof}
We already know by \cite[3.2]{exends} that $A_t$ is a unital
$C^*$-subalgebra of $A$ with unit $V_t(1_A)=1_A$, and that $\gamma_t$
is an isomorphism, $\forall t\in G$. Since $V$ is a partial
representation we have that $\gamma_e=V_e=Id$. 
Suppose now that $c$ belongs to the domain of 
$\gamma_s\gamma_t$, that is, $c\in A_{t^{-1}}$ is such that
$\gamma_t(c)\in A_{s^{-1}}$.  
Then $\gamma_s\gamma_t(c)\in A_s$ and
$\gamma_s\gamma_t(c)=V_sV_t(V_{t^{-1}}(\gamma_t(c))) 
=V_{st}(V_{t^{-1}}(\gamma_t(c)))=V_{st}(c)\in A_{st}$. Then
$\gamma_s\gamma_t(c)\in A_s\cap A_{st}$, and we may apply
$\gamma_{t^{-1}s^{-1}}$ to $\gamma_s\gamma_t(c)$. Since $V$ is a
partial action we obtain:  
\[\gamma_{t^{-1}s^{-1}}\gamma_s\gamma_t(c)
=V_{t^{-1}s^{-1}}V_s\gamma_t(c)
=V_{t^{-1}}V_{s^{-1}}V_s\gamma_t(c)
=\gamma_{t^{-1}}\gamma_{s^{-1}}\gamma_s\gamma_t(c)=c, \]
whence $\gamma_{st}(c)=\gamma_s\gamma_t(c)$. This shows that
$\gamma_{st}$ extends $\gamma_s\gamma_t$, $\forall s,t\in G$, and
therefore $\gamma$ is a partial action. 
\end{proof}

\par Observe that if $V$ is an interaction group of the type
considered in \ref{prop:1}, then $E_r=FF_r|_A$, and
$E_rE_s=FF_rF_s|_A$ (with the notations of \ref{prop:1} and
\ref{prop:pa}).  
\bigskip
\par The usual notion of partial actions of groups on $C^*$-algebras  
requires that the domains of the partial automorphisms are
ideals. In the commutative case, partial actions on a $C^*$-algebra 
correspond exactly with partial actions on the spectrum of the
algebra, where the domains of the partial homeomorphisms are open
subsets of the spectrum (\cite[Proposition~1.5]{fpoids}). Instead,
partial actions on unital commutative $C^*$-algebras as the ones
considered in~\ref{prop:pa} lead to a different notion of partial
action on a topological space. In fact, let $A=C(X)$ be a unital
commutative $C^*$-algebra, and let $\gamma=(\{A_t\},\{\gamma_t\})$ be
a partial action of~$G$ on~$A$, where each $A_t$ is a unital
subalgebra of $A$, with the same unit. Then the dual notion of the
partial action $\gamma$ should be expressed in terms of the spectra of
the subalgebras $A_t$ and the maps induced by $\gamma$ between them.
Although we will not give here the exact conditions that such a
collection of spaces and maps must satisfy, it is clear that the result
is not a partial action in the usual sense, as the spectrum of $A_t$
is not a subspace but a quotient of $X$. 
 
\subsection{Dilations of interaction groups}
We introduce next the notion of dilation of an interaction group $V$,
and we study its relation with the partial action associated with
$V$.  
\begin{defn}\label{defn:dil}
Let $V:G\to B(A)$ be an interaction group. 
A dilation of $V$ is a pair $(i,T)$, where
$T=(B,\beta,F)\in\mathcal{T}_G$ and $i:A\to B$ is a homomorphism of
$C^*$-algebras such that $iV_t=F\beta_ti$, $\forall t\in G$. If
$B=\overline{\textrm{span}}\{\beta_ti(a):\,a\in A,t\in G\}$, we say
that the dilation is minimal. The dilation is called faithful if so is
$F$, and it is called admissible if $T\in\mathcal{D}_G$ (recall that a
positive map $F$ is called faithful when $b\neq 0$ implies
$F(b^*b)\neq 0$).    
\end{defn}
\begin{prop}\label{prop:restriction}
Let $V:G\to B(A)$ be an interaction group, and suppose that 
$(i,T)$ is a minimal dilation of $V$, where $T=(B,\beta,F)$. Then we
have $F((FF_t-F_tF)(b)^*(FF_t-F_tF)(b))=0$, $\forall b\in B$.   
\end{prop}
\begin{proof}
We must show that
$F\beta_tF\beta_{t^{-1}}(b)-\beta_tF\beta_{t^{-1}}F(b)$ belongs to the
left ideal $L_F:=\{b\in
B:\, F(b^*b)=0\}$ of $B$, $\forall  b\in B$. Since $B$ is the closed
linear span of the set $\cup_{s\in G}\beta_si(A)$, it is enough to
prove that
$F\beta_tF\beta_{t^{-1}}(\beta_si(a))-\beta_tF\beta_{t^{-1}}F(\beta_si(a))\in 
L_F$, that is, $V_tV_{t^{-1}s}(a)-\beta_tV_{t^{-1}}V_s(a)\in L_F$,
$\forall s\in G$, $a\in A$. Since $F$ is an
$A$-bimodule map which is the identity operator on $A$, and
since $F\beta_t|_A=V_t$, we have that the expression 
$F\big((V_tV_{t^{-1}s}(a)-\beta_tV_{t^{-1}}V_s(a))^*
(V_tV_{t^{-1}s}(a)-\beta_tV_{t^{-1}}V_s(a))\big)$ is equal to 
$V_tV_{t^{-1}s}(a^*)\big(V_tV_{t^{-1}s}(a)-V_tV_{t^{-1}}V_s(a)\big)
-V_tV_{t^{-1}}V_s(a^*)\big(V_tV_{t^{-1}s}(a)-V_tV_{t^{-1}}V_s(a)\big)$,
which is zero because $V$ is a partial representation.        
\end{proof}
\begin{cor}\label{cor:minfaithful}
Any minimal and faithful dilation of an interaction
group is admissible.
\end{cor}
\par Suppose that $\beta$ is an action of $G$ on the $C^*$-algebra
$B$, and that $A$ is a $C^*$-subalgebra of $A$. The restriction of
$\beta$ to $A$ is the partial action $\beta|_A:=(\{A'_t\}_{t\in G},
\{\gamma'_t\}_{t\in G})$, where $A'_t:=A\cap\beta_t(A)$ and
$\gamma'_t(a):=\beta_t(a)$, $\forall a\in A'_{t^{-1}}$, $t\in
G$. In case that the $C^*$-algebra generated by $\{\beta_t(a):\, a\in
A, t\in G\}$ is all of~$B$, we say that $\beta$ is an
\textbf{enveloping action} for~$\gamma'$.    

\begin{prop}\label{prop:env}
Suppose that $V:G\to B(A)$ is an interaction group with dilation
$(i,(B,\beta,F))$, where $A$ is a $C^*$-subalgebra of $B$ and $i:A\to
B$ is the natural inclusion.  Let $\gamma$ be the partial action of
$G$ on $A$ given by Proposition~\ref{prop:pa}, and let 
$\gamma':=\beta|_A$. Then $A_t\supseteq A'_t:=A\cap \beta_t(A)$ and
$\gamma_t(a)=\gamma'_t(a)$, $\forall t\in G$, $a\in A'_{t^{-1}}$. If
the dilation is admissible then $\gamma=\beta|_A$. In particular if
the dilation is faithful then $\gamma$ is the restriction
of $\beta$ to $A$.  
\end{prop}
\begin{proof}
If $a\in A$, then $a\in A'_{t^{-1}}\iff\beta_t(a)\in A\iff
\beta_t(a)=F\beta_t(a)\iff \beta_t(a)=V_t(a)$. Then if $a\in
A'_{t^{-1}}$ we have $\gamma'_t(a)=\beta_t(a)=V_t(a)\in A_t$, which
shows that $A'_t\subseteq A_t$ and
$\gamma_t'=V_t|_{A_{t^{-1}}'}=\gamma_t|_{A_{t^{-1}}'}$. On the other
hand, if $FF_t=F_tF$, then if $a\in A_t$ we have:
\[a=V_tV_{t^{-1}}(a)
=FF_t(a)=F_tF(a)=\beta_t(F\beta_{t^{-1}}F(a))\in
A\cap\beta_t(A)=A_t',\]   
whence $A_t=A_t'$, and $\gamma_t=\gamma_t'$. 
The last two assertions follow from
Proposition~\ref{prop:restriction} and
Corollary~\ref{cor:minfaithful}.     
\end{proof}
\begin{cor}\label{cor:enveloping}
Suppose that $V:G\to B(A)$ is an interaction group with admissible
dilation $(i,(B,\beta,F))$, where $i:A\to B$ is an embedding (i.e.:
$i$ is injective). Then the restriction of $\beta$ to 
$C:=\overline{span}\{\beta_ti(a):\,t\in G, a\in A\}$ is an enveloping
action for the partial action $\gamma$ of $G$ on $A$ given by
Proposition~\ref{prop:pa}. In particular, if the dilation is minimal
then $\beta$ is an enveloping action for $\gamma$.  
\end{cor}
\section{The dilation} 
A cancelative monoid $P$ is called an Ore
semigroup if $Pr\cap Ps\neq \emptyset$, $\forall r,s\in P$. It follows
by induction that $P$ is an Ore semigroup if and only if 
$P_{t_1}\cap\ldots\cap P_{t_n}\neq\emptyset$, $\forall
t_1,\ldots,t_n\in P$. Then $P$ is partially ordered by the relation
$r\leq s\iff s\in Pr$ (equivalently: $r\leq s\iff Pr\supseteq Ps$),
and it is even directed by that relation.  
\par Any cancelative abelian monoid $P$ is an  Ore semigroup. In fact,
such a monoid embeds in its Grothendieck group $G$, and every element
$t\in G$ can be written as $t=v^{-1}u$, with $u,v\in P$. Therefore, if
$r,s\in P$, writing $rs^{-1}=u^{-1}v$, with $u,v\in P$, gives
$t:=ur=vs\in Pr\cap Ps$, so $P$ is an Ore semigroup (and $P\ni t\geq
r,s$). More generally, we have the following theorem \cite[Theorem
1.1.2]{ml}, which shows that there is a functor from the category of
Ore semigroups into the category of groups:  
\begin{thm}[Ore, Dubreil]\label{thm:oredubreil}
A semigroup $P$ can be embedded in a group $G$ with $P^{-1}P=G$ if and
only if it is an Ore semigroup. In this case the group $G$ is
determined up to canonical isomorphism and every semigroup
homomorphism $\phi$ from $P$ into a group $H$ extends uniquely to a
group homomorphism $\varphi:G\to H$.  
\end{thm}
\par In case $P$ is an Ore semigroup we say that the group $G$ in
\ref{thm:oredubreil} is the enveloping group of $P$.   
\par A key ingredient in our process of dilating the 
interaction groups under consideration is Laca's theorem
\cite[2.1.1]{ml}. For the convenience of the reader we recall it
below:    
\begin{thm}[M. Laca, \cite{ml}]\label{thm:ml}
Assume $P$ is an Ore semigroup with enveloping group $G=P^{-1}P$ and
let $\alpha$ be an action of $P$ by unital injective endomorphisms of
a unital $C^*$-algebra $A$. Then there exists a $C^*$-dynamical system 
$(B,G,\beta)$, unique up to isomorphism, consisting of an action
$\beta$ of $G$ by automorphisms of a $C^*$-algebra $B$ and an
embedding $i:A\to B$ such that: 
\begin{enumerate}
\item $\beta$ dilates $\alpha$, that is, $\beta_t\circ
  i=i\circ\alpha_t$, for $t$ in $P$, and
\item $(B,G,\beta)$ is minimal, that is, $\bigcup_{t\in
    P}\beta_t^{-1}(i(A))$ is dense in $B$. 
\end{enumerate}
\end{thm} 
\par Note that $i$ is unital:
\[\beta_{t^{-1}}i(a)i(1_A)=\beta_{t^{-1}}(i(a)\beta_t(i(1_A)))
=\beta_{t^{-1}}(i(a\alpha_t(1_A)))=i(a),\, \forall t\in P,\] so taking
adjoints and recalling that 
$\{\beta_{t^{-1}}(i(a)):\ t\in P,a\in A\}$ is dense in $B$, we see
that $i(1_A)=1_B$.\\ 
     
\par From now on $G$ will denote the enveloping group of the Ore
semigroup $P$. 

\begin{lem}\label{lem:restriction}
Let $\alpha$ be an action of the Ore semigroup $P$ by unital injective
endomorphisms of the unital $C^*$-algebra $A$, and suppose that
$V:G\to B(A)$ is an interaction group such that $V|_P=\alpha$. If 
$(i,(B,\beta,F))$ is an admissible dilation of $V$, then
$\beta_ti=iV_t=i\alpha_t$, $\forall t\in P$. 
\end{lem}
\begin{proof}
Note first that for $t\in G$:
$FF_ti=F\beta_t(F\beta_{t^{-1}}i)=F\beta_tiV_{t^{-1}}=iV_tV_{t^{-1}}$.
If now $t\in P$ we have $V_{t^{-1}}\alpha_t=id_A$, and therefore
\[\beta_ti=\beta_tiV_{t^{-1}}\alpha_t
            =\beta_tF\beta_{t^{-1}}i\alpha_t
            =F_tFi\alpha_t
            =FF_ti\alpha_t
            =iV_tV_{t^{-1}}V_t=iV_t. \] 
\end{proof}
\begin{thm}\label{thm:main}
Let $\alpha$ be an action of the Ore semigroup $P$ by unital injective
endomorphisms of the unital $C^*$-algebra $A$, and suppose that
$V:G\to B(A)$ is an interaction group such that $V|_P=\alpha$. Then
$V$ has a minimal admissible dilation $(i,T)$, where $T=(B,\beta,F)$
and $i:A\to B$ is an embedding, which has the
following universal property. If $(i',(B',\beta',F'))$ is another
admissible dilation of $V$, then there exists a unique morphism
$\phi:(B,\beta,F)\to (B',\beta',F')$ such that $\phi i=i'$. Therefore
the dilation $(i,T)$ is unique up to isomorphism in the class of
minimal and admissible dilations.  
\end{thm}
\begin{proof}
Let $i:(A,\alpha)\to (B,\beta)$ be the minimal dilation of
$(A,\alpha)$ provided by Laca's theorem. We suppose, as we can do,
that $i$ is the natural inclusion, so $A\subseteq B$. We proceed next
to define a conditional expectation $F:B\to A$. To this end note   
first that if $r,s\in P$, with $r\leq s$, and $a_r,a_s\in A$ are such
that $\beta_{r^{-1}}(a_r)=\beta_{s^{-1}}(a_s)$, then
$\beta_{sr^{-1}}(a_r)=a_s$, so $\alpha_{sr^{-1}}(a_r)=a_s$ by
\ref{lem:restriction}. Therefore 
\[\mathcal{L}_s(a_s)
=\mathcal{L}_s\alpha_{sr^{-1}}(a_r)
=V_{s^{-1}}V_{sr^{-1}}(a_r)
=\mathcal{L}_s\alpha_s\mathcal{L}_r(a_r)
=\mathcal{L}_r(a_r).\] 
Thus we may define $F_0:\bigcup_{t\in P}\beta_{t^{-1}}(A)\to B$ such
that $F_0(b)=\mathcal{L}_{t}(\beta_t(b))$, $\forall b\in
\beta_{t^{-1}}(A)$. Since
$\norm{F_0(b)}=\norm{\mathcal{L}_{t}(\beta_t(b))}\leq\norm{b}$, $F_0$
extends uniquely to a bounded operator $F:B\to A$, which is easily
seen to be positive and to satisfy $F^2=F$ and $F(B)=A$. Then $F$ is a
conditional expectation with range $A$. We claim that $(B,\beta,F)$ is
a minimal admissible dilation of $V$. In fact, if $t\in G$
and $r,s\in P$ are such that $t=r^{-1}s$, then
\[F\beta_t|_A
=F\beta_{r^{-1}}\beta_{rt}|_A
=F\beta_{r^{-1}}\alpha_s
=\mathcal{L}_r\alpha_s
=V_{r^{-1}}V_rV_{r^{-1}s}
=V_{r^{-1}s}
=V_t.\] 
Since $\bigcup_{t\in P}\beta_{t^{-1}}(A)$ is dense in $B$ we have that
$(B,\beta,F)$ is minimal, and to see that it is also admissible, it is
enough to check that $FF_t\beta_{r^{-1}}|_A=F_tF\beta_{r^{-1}}|_A$,
$\forall t\in G, r\in P$. On the one hand we have 
\begin{equation}\label{eqn:fft}
FF_t\beta_{r^{-1}}|_A
=F\beta_tF\beta_{t^{-1}r^{-1}}|_A
=V_tFV_{t^{-1}r^{-1}}
=E_tV_{r^{-1}}
\end{equation}
On the other hand, let $t\in G$, $t=u^{-1}v$, $u,v\in
P$. Using Lemma~\ref{lem:restriction} and recalling that
$E_uE_v=E_vE_u$, we have  
\begin{multline}\label{eqn:ftf}
\begin{split}
F_tF\beta_{r^{-1}}|_A
=\beta_{u^{-1}}\beta_vF\beta_{v^{-1}u}V_{r^{-1}}
=\beta_{u^{-1}}V_vV_{v^{-1}u}V_{r^{-1}}\\
=\beta_{u^{-1}}E_vE_uV_{u}V_{r^{-1}}
=\beta_{u^{-1}}E_uE_vV_{u}V_{r^{-1}}\\
=\beta_{u^{-1}}V_uV_{u^{-1}}V_vV_{v^{-1}}V_{u}V_{r^{-1}}\\
%=\beta_{u^{-1}}V_uV_{u^{-1}}V_vV_{v^{-1}u}V_{r^{-1}}
%=\beta_{u^{-1}}V_uV_{u^{-1}v}V_{v^{-1}u}V_{r^{-1}}\\
=\beta_{u^{-1}}\beta_uV_{t}V_{t^{-1}}V_{r^{-1}}
=E_tV_{r^{-1}}
\end{split}
\end{multline}
From (\ref{eqn:fft}) and (\ref{eqn:ftf}) we conclude that
$(B,\beta,F)$ is admissible. 
We see next that  $(B,\beta,F)$ has the claimed universal
property. Then suppose that $(i',(B',\beta',F'))$ is another admissible 
dilation of $V$. By Lemma~\ref{lem:restriction} we have that
$\beta'|_P=i'\alpha$, and then by the universal property of the pair
$(B,\beta)$ there exists a unique homomorphism $\phi:B\to B'$ such
that $\phi i=i'$ and $\beta_t'\phi=\phi\beta_t$ $\forall t\in G$. In
particular $\phi\beta_{r^{-1}}i=\beta_{r^{-1}}'\phi
i=\beta_{r^{-1}}'i'$, $\forall r\in P$. Thus
\[F'\phi\beta_{r^{-1}}i=F'\beta_{r^{-1}}'i'=i'V_{r^{-1}}=
\phi iV_{r^{-1}}=\phi F\beta_{r^{-1}}i,\ \  \forall r\in P.\] The 
equality $\phi F=F'\phi$ follows now from the density of  
$\bigcup_{r\in P}\beta_{r^{-1}}i(A)$ in $B$ and the continuity of the
involved maps.  
\end{proof}
\begin{rem} Suppose $V$ and $V'$ are interaction groups that extend actions by
injective unital endomorphisms of the Ore semigroup $P$. Suppose as
well that $\psi:(A,V)\to (A',V')$ is a morphism of 
interaction groups, and let $(i,T)$ and $(i',T')$ be the corresponding
minimal admissible dilations of $V$ and $V'$. Then $(i'\psi,T')$ is an
admissible dilation of $V$, so there exists a unique morphism
$\phi:T\to T'$ such that $\phi i=i'\psi$. In this way we obtain a
functor from the category of interaction groups that extend actions by
injective unital endomorphisms of the Ore semigroup $P$ into the
category $\mathcal{D}_G$, where $G$ is the enveloping group of~$P$.
\end{rem}   
\par We end the section with a result concerning enveloping
actions.   
\begin{prop}\label{prop:env2} 
Let $V$ be an interaction group like in \ref{thm:main}, and let
$\gamma$ be the partial action associated to $V$ via \ref{prop:pa}.  
Then $\gamma$ has an enveloping action, which is unique up to
isomorphism.  
\end{prop}
\begin{proof}
It follows from \ref{thm:main} and \ref{cor:enveloping} that the
action $\beta$ provided by Theorem~\ref{thm:main} is an enveloping
action for $\gamma$. Suppose now that $\beta':G\times B'\to B'$ is
another enveloping action for $\gamma$, where $B'$ is a $C^*$-algebra
which contains $A$. To show that $\beta$ and $\beta'$ are isomorphic,
it is enough to show that $\beta'$ satisfies properties 1. and 2. of
Theorem~\ref{thm:ml}. It is clear that $\beta'$ satisfies the first
property, so let us see that it also verifies the second one. Note
that if $t=r^{-1}s\in G$, with $r,s\in P$, then 
$\beta'_t(A)=\beta'_{r^{-1}}\alpha_s(A)\subseteq
\beta'_{r^{-1}}(A)$. On the other hand, suppose $r,s\in P$, with 
$r\leq s$. Then, since $sr^{-1}\in P$, we have
$A\supseteq\alpha_{sr^{-1}}(A)=\beta'_s\beta'_{r^{-1}}(A)$, so
$\beta'_{s^{-1}}(A)\supseteq \beta'_{r^{-1}}(A)$. Thus
$\beta'_t(A)\subseteq \bigcup_{r\in P}\beta'_{r^{-1}}(A)$, $\forall
t\in G$, because $P$ is directed by its partial order. This implies
that $B'$ is the closure of $\bigcup_{r\in P}\beta'_{r^{-1}}(A)$, as
we wanted to prove.    
\end{proof}
  
\section{Dilations of Exel-Renault interaction groups} 
\par In this section we specialize to certain interaction groups
occuring on commutative $C^*$-algebras. More precisely, we are
interested in the interaction groups studied in \cite{exren}. In that
work, the authors considered right actions  
$\theta:P\times X\to X$, where $P$ is an Ore semigroup with enveloping
group $G$, and $\theta_t$ is an onto local homeomorphism of the
compact Hausdorff space $X$, that is, $\theta_t:X\to X$ is a covering
map. Dualizing, $\theta$ induces a left action  
$\alpha$ of $P$ by injective unital endomorphisms of $A=C(X)$. It is
shown in \cite{exren} that for $\alpha$ to be extended to an
interaction group $V:G\to B(A)$ it is enough that there exists a
certain map $\omega:P\times X\to [0,1]$, associated to $\theta$. This 
map is called a cocycle and is determined by the fact that
$E_t(a)(x)=\sum_{\theta_t(y)=\theta_t(x)}\omega(t,y)a(y)$, $\forall
t\in P$, $a\in A$ and $x\in X$, where $E_t=V_tV_{t^{-1}}$. In this
case Theorem~\ref{thm:main} can be applied, so one concludes that the 
interaction groups considered by Exel and Renault in \cite{exren} have
minimal admissible dilations. We mention in passing that for these
interaction groups Theorem~\ref{thm:main} could be proved by
using exclusively measure-theoretic arguments, but we will not do it
here. The aim of this section is to show that the minimal admissible
dilations of the Exel-Renault interaction groups are also faithful.

\subsection{Conditional expectations on commutative %unital
  $C^*$-algebras.} 
\par We begin by giving a characterization of conditional expectations   
from a commutative unital $C^*$-algebra onto a unital
$C^*$-subalgebra, suitable for our purposes. We also describe the 
transfer operators for an endomorphism induced by a covering map. For
more information about conditional expectations we 
refer the reader to \cite{w} and \cite{blanchard}. 
\par We fix a notation we will use until the end of the present
paragraph. Let $B=C(Z)$ be a unital $C^*$-algebra and $A=C(X)$ a
unital $C^*$-subalgebra of~$B$. Note that $X$ is  
homeomorphic to the quotient space of $Z$ with respect to the relation
$z\sim z'\iff a(z)=a(z')$, $\forall a\in A$. Let $\pi:Z\to X$ be the
corresponding quotient map. Observe that an element $a\in C(X)$, when
seen as an element of~$B$, sends $z\in Z$ into $a(\pi(z))$. Denote by
$P(Z)$ the set of regular Borel probability measures on~$Z$.   

\begin{prop}\label{prop:positivemaps}
With the above notation, let $F:B\to A$ be a unital linear map. Then
$F$ is positive if and only if there exists a map $\mu:X\to
P(Z)$ that is $w^*$-continuous and such
that 
\begin{equation}\label{eqn:positive}
F(b)(x)=\int_Zb(z)d\mu_x(z),\ \forall b\in B, x\in X.  
\end{equation}
Equation \eqref{eqn:positive} establishes a bijective correspondence 
between unital positive linear maps $F:B\to A$ and $w^*$-continuous
maps $\mu:X\to P(Z)$.   
\end{prop}
\begin{proof}
Let $\epsilon_x:A\to \C$ be evaluation in $x\in X$. Then if $F$ is
positive $\epsilon_x\circ F$ is a state of $B$. Let $\mu_x$ be the
probability measure provided by the Riesz-Markov representation
theorem, such that $\epsilon_x\circ F(b)=\int_Zb(z)d\mu_x(z)$, $\forall b\in
B$. Since $F(b)$ is a continuous function defined on $X$, it follows
that $x\mapsto\mu_x$ is $w^*$-continuous. Conversely, it is clear that
if a $w^*$-continuous map $\mu:X\to P(Z)$ is such that 
$F(b)(x)=\int_Zb(z)d\mu_x(z),\ \forall b\in B, x\in X$, then $F(b)$ is
positive whenever $b$ is positive. Finally, it is obvious that the
correspondence $F\mapsto \mu$ is one to one and onto.     
\end{proof}

\par It is clear that in Proposition~\ref{prop:positivemaps} above $A$
does not need to be a subalgebra of $B$. 
\begin{exmp}\label{exmp:endo}
Suppose $\alpha:B\to B$ is an injective unital endomorphism and let
$A:=\alpha(B)$. Then there exists a homeomorphism 
$\bar{\xi}:X\to Z$ such that $\alpha(b)=b\circ\bar{\xi}\in
A$. Therefore
$\alpha(b)(x)=b(\bar{\xi}(x))=\int_Zb(z)d\delta_{\bar{\xi}(x)}$, where 
$\delta_z$ denotes the Dirac measure concentrated at $z$. Thus the map 
$\mu$ provided by \ref{prop:positivemaps} for $\alpha$ is given by: 
$\mu_x=\delta_{\bar{\xi}(x)}$.      
\end{exmp}

\begin{prop}\label{prop:charcondexp}
A linear map $F:B\to A$ is an onto conditional expectation
if and only if there exists a map $\mu:X\to P(Z)$ such that $\mu$ is
$w^*$-continuous, $\textrm{supp}(\mu_x)\subseteq\pi^{-1}(x)$, $\forall
x\in X$, and \[F(b)(x)=\int b(z)d\mu_x(z), \forall b\in B, x\in X.\]
If there exists such a map $\mu$, then it is unique, and $F$ is
faithful if and only if the interior of the set $Z_\mu:=\{z\in Z:
z\notin\textrm{supp}(\mu_{\pi(z)})\}$ is empty.
 Consequently, if $\textrm{supp}(\mu_x)=\pi^{-1}(x)$, $\forall x\in
 X$, then $F$ is faithful.     
\end{prop}  
\begin{proof}
Suppose first that there exists such a map $\mu$. If $a\in A,b\in B$
and $x\in X$:
\begin{equation*}\begin{split}
  F(ab)(x)&=\int_{\pi^{-1}(x)}a(\pi(z))b(z)d\mu_x(z)
=a(x)\int_{\pi^{-1}(x)}b(z)d\mu_x(z)\\&=(aF(b))(x).\end{split}\end{equation*} 
Then $F(ab)=aF(b)$. A similar computation shows that $F(a)=a$, $\forall
a\in A$, whence $F$ is a conditional expectation. Conversely, suppose
that $F:B\to A$ is a conditional expectation, and let $\mu:X\to P(Z)$
be the map provided by Proposition~\ref{prop:positivemaps} for the
unital positive map $F$. Let us see 
that $\textrm{supp}(\mu_x)\subseteq\pi^{-1}(x)$. Suppose
$z\notin\pi^{-1}(x)$. Then $\pi(z)\neq x$, so there exist open
disjoint sets $V_z$ and $V_x$ in $X$ such that $\pi(z)\in V_z$ and
$x\in V_x$. Let $a\in A$ be such that $a(X)=[0,1]$, with
$\textrm{supp}(a)\subseteq V_x$ and $a(x)=1$. Then, since $a=F(a)$, 
$a(x)=1$, and $\textrm{supp}(a\circ\pi)\subseteq\pi^{-1}(V_x)$:  
\[1=F(a)(x)
=\int_Za\circ\pi\, d\mu_x
=\int_{\pi^{-1}(V_x)}\hspace*{-2ex}
a\circ\pi\, d\mu_x
\leq\mu_x(\pi^{-1}(V_x))\leq 1. \]    
It follows that $\mu_x(Z\setminus\pi^{-1}(V_x))=0$. 
Then $\pi^{-1}(V_z)\cap\pi^{-1}(V_x)=\emptyset$, hence
$\mu_x(\pi^{-1}(V_z))=0$. Since $\pi^{-1}(V_z)$ is open we have that
$\pi^{-1}(V_z)\cap\textrm{supp}(\mu_x)=\emptyset$. This shows that
$z\notin \textrm{supp}(\mu_x)$ and therefore
$\textrm{supp}(\mu_x)\subseteq\pi^{-1}(x)$.    
\par Suppose now that there exists a non--empty open subset $V$ of $Z$
such that $z\notin\textrm{supp}(\mu_{\pi(z)})$, $\forall z\in V$. Let
$b\in B^+$ be a non--zero element such that
$\textrm{supp}(b)\subseteq~V$. Then for all $x\in X$ we have 
$F(b)(x)=\int_{\pi^{-1}(x)\cap 
  V}b(z)d\mu_x(z)=0$ since $\pi^{-1}(x)\cap
V\cap\textrm{supp}(\mu_{x})=\emptyset$. Thus $F$ is not faithful.  
Conversely, if $F$ is not faithful, let $0\neq b\in B^+\cap\ker F$, and
$V\subseteq \textrm{supp}(b)$ such that $b(z)\geq\delta$, for some
positive $\delta$ and for all $z\in V$. Then, if $z_0\in V$ and
$x=\pi(z_0)$ we have 
\[0=F(b)(x)=\int_{\pi^{-1}(x)}\hspace*{-2ex}b\,d\mu_{x}\geq
\delta\mu_{x}(\pi^{-1}(x)\cap\textrm{supp}(b))\geq
\delta\mu_{x}(V\cap\pi^{-1}(x)).\] This shows that
$z_0\notin\textrm{supp}(\mu_{\pi(z)})$, $\forall z_0\in V$.  
\end{proof}
\begin{cor}\label{cor:transfer}
Let $\xi:Z\to Z$ be an onto continuous map and $\alpha:B\to B$
its dual map. Let $A=C(X)$ be the range of $\alpha$ and $\pi:Z\to X$ 
the canonical projection. Then a map $\mathcal{L}:B\to B$ is a 
transfer operator for $\alpha$ if and only if there exists a 
$w^*$-continuous map $\nu:Z\to P(Z)$ such that 
\begin{equation}\label{eq:transfer}
\mathcal{L}(b)(z)=\int_{\xi^{-1}(z)}b(u)d\nu_z(u)
\end{equation} 
with $\supp(\nu_z)\subseteq\xi^{-1}(z)$, $\forall z\in Z$. In this
case the map $\nu$ is unique. More
precisely, if $\mathcal{L}$ is a transfer operator, then
$\nu_z=\mu_{\pi(z')}$, where $\mu$ is the map associated by 
\ref{prop:charcondexp} to the conditional expectation
$\alpha\mathcal{L}$, and $z'$ is any element of~$\xi^{-1}(z)$.   
\end{cor}
\begin{proof}
Suppose that $\mathcal{L}:B\to B$ is a
transfer operator for $\alpha$. Then $F:=\alpha\mathcal{L}$ is a
conditional expectation onto $A$. By
Proposition~\ref{prop:charcondexp} we have 
$F(b)(x)=\int_{\pi^{-1}(x)}b(u)d\mu_x(u)$, for a unique
$w^*$-continuous map $\mu:X\to P(Z)$. Consequently we have 
$F(b)(z)=\int_{\pi^{-1}(\pi(z))}b(u)d\mu_{\pi(z)}(u)$. Since $\xi$ is
onto, for $z\in Z$ there exists $z'$ such that 
$z=\xi(z')$. Then, as $F=\alpha\mathcal{L}$, we get: 
$\mathcal{L}(b)(z)=\mathcal{L}(b)(\xi(z'))
=F(b)(z')
=\int_{\pi^{-1}(\pi(z'))}b(u)d\mu_{\pi(z')}(u)
%=\int_{\xi^{-1}(\xi(z'))}b(u)d\mu_{\pi(z')}(u)
=\int_{\xi^{-1}(z)}b(u)d\mu_{\pi(z')}(u)$, $\forall b\in 
B$, $z'\in\xi^{-1}(z)$. So if $\nu:Z\to P(Z)$ is the map associated 
to the unital positive map $\mathcal{L}$ on $B$
by~\ref{prop:positivemaps}, we have, for $z,z'\in Z$, with
$\xi(z')=z$:   
$\mathcal{L}(b)(z)
=\int_{\xi^{-1}(z)}b(u)d\mu_{\pi(z')}(u)
=\int_Zb(u)d\nu_{z}(u)$, and therefore
$\nu_{z}=\mu_{\pi(z')}$, $\forall z'\in\xi^{-1}(z)$. In
particular, if $\xi(z')=z$, then: 
$\supp(\nu_z)=\supp\mu_{\pi(z')}\subseteq\pi^{-1}(\pi(z')) 
=\xi^{-1}(\xi(z'))=\xi^{-1}(z)$. Conversely, it is readily checked
that a map $\mathcal{L}$ given by \eqref{eq:transfer} for such a map
$\nu$ is a transfer operator for $\alpha$.   
\end{proof}
\par As an immediate consequence of the above result we have
\begin{cor}\label{cor:transfer2}
Let $\alpha$ be as in Corollary~\ref{cor:transfer}. Then the map
$\mathcal{L}\mapsto \alpha\mathcal{L}$ is a bijective correspondence
between the sets of transfer operators for $\alpha$ and conditional   
expectations onto $\alpha(B)$. Moreover $\alpha\mathcal{L}$ is
faithful if and only if so is $\mathcal{L}$. 
\end{cor}
\begin{proof}
The last assertion, which is not implied by \ref{cor:transfer},
follows from the injectivity of $\alpha$.  
\end{proof}
 When the quotient map $\pi:Z\to X$ is a covering map we can
 be more precise:
\begin{cor}\label{cor:contcoc}
Suppose that the quotient map $\pi:Z\to X$ is a covering map. Then a
linear map $F:B\to A$ is an onto conditional 
expectation if and only if there exists a continuous map 
$\omega:Z\to [0,1]$ such that $\sum_{z\in\pi^{-1}(x)}\omega(z)=1$,
$\forall x\in X$, and $F(b)(x)=\sum_{z\in\pi^{-1}(x)}\omega(z)b(z)$,
$\forall b\in B$, $x\in X$. In this case the map $\omega$ is unique,
and $F$ is faithful if and only if the set $\omega^{-1}(0)$ is nowhere
dense.    
\end{cor}
\begin{proof}
If $x\in X$, then $\pi^{-1}(x)$ is finite, because $\pi$ is
a local homeomorphism and $Z$ is compact. Thus a map $\mu:X\to P(Z)$ 
such that $\supp(\mu_x)\subseteq\pi^{-1}(x)$ is nothing but a map
$\omega:Z\to [0,1]$ such that $\sum_{z\in\pi^{-1}(x)}\omega(z)=1$ and 
$\mu_x=\sum_{z\in\pi^{-1}(x)}\omega(z)\delta_z$, $\forall x\in
X$. Then, if $\omega$ is continuous, $\mu$ is
$w^*$-continuous. Suppose conversely that $\mu$ is $w^*$-continuous. 
Fix $z_0\in Z$, and let $V$ be an open neighborhood of $z_0$ on which
the restriction of $\pi$ is a homeomorphism onto its image. Let $b\in
C(Z)$ be such that $\textrm{supp}(b)\subseteq V$, and $b=1$ on a
neighborhood $U$ of $z_0$. If $z\in U$: 
\[\omega(z)=\omega(z)b(z)
=\sum_{z'\in\pi^{-1}(\pi(z))}\omega(z')\int_Z
b\,d\delta_{z'}
=\int_Zb\,d\mu_{\pi(z)}
\]
Since $\mu$ is $w^*$-continuous and $\pi$ is continuous, it follows
that $\omega$ also is continuous because:  
\[\lim_{z\to z_0}\omega(z)=\lim_{z\to z_0}\int_Zb\,d\mu_{\pi(z)}
                         =\int_Zb\,d\mu_{\pi(z_0)}=\omega(z_0)
\]
\par Note finally that if $z\in\pi^{-1}(x)$, then $z\in\supp(\mu_x)$
if and only if $\omega(z)\neq 0$. The proof now follows by combining
the considerations above with Proposition~\ref{prop:charcondexp}.  
\end{proof}
\par If, in the situation of Corollary~\ref{cor:contcoc}, the map
$\omega$ exists and is positive, it follows from
\cite[Proposition~2.8.9]{w} that the associated conditional
expectation $F$ is of index-finite type (in the sense of Watatani,
\cite{w}), and moreover $\textrm{Index}\,F(z)=1/\omega(z)$, $\forall 
z\in Z$. In particular $F$ is faithful, a fact that also follows
from~\ref{prop:charcondexp} and~\ref{cor:transfer2}. Conversely, if
$F$ is of index-finite type, then 
$\textrm{Index}\,F$ is positive (\cite[Lemma~2.3.1]{w}) and 
$\omega(z)=1/\textrm{Index}\,F(z)$.  
\begin{exmp}
Consider the covering map $\theta:S^1\to S^1$
  given by $\theta(z)=z^2$, and let $\alpha:B\to B$ be its dual map,
  where $B=C(S^1)$. Then $A:=\alpha(B)=\{b\in B:b(z)=b(-z),\,\forall 
  z\in S^1\}$. Consider any continuous function $\omega':[0,\pi]\to
  [0,1]$ such that $\omega'(\pi)=1-\omega'(0)$, and let
  $\omega:[0,2\pi]\to[0,1]$ be the extension of $\omega'$ such that 
  $\omega(t)=1-\omega'(t-\pi)$, $\forall t\in(\pi,2\pi]$. Since
  $\omega$ is continuous and $\omega(0)=\omega(2\pi)$, we can look at
  $\omega$ as a continuous map from $S^1$ into $[0,1]$.  
  It is clear that $\sum_{z^2=z_0}\omega(z)=1$, $\forall z_0\in
  S^1$. Therefore by~\ref{cor:contcoc}
  $\omega$ defines a conditional expectation $F_\omega$. Consider the
  construction above for the following three cases:
  $\omega'_1(t)=\frac{1}{2}$, $\omega'_2(t)=\frac{t}{\pi}$ and
  $\omega'_3(t)=\begin{cases}0&\textrm{ if }t\in[0,\frac{\pi}{2}]\\
                             \frac{2t}{\pi}-1&\textrm{ if
                             }t\in[\frac{\pi}{2},\pi]\end{cases}$.
  Then $F_{\omega_1}$ is of index-finite type, $F_{\omega_2}$
  is faithful but not of index-finite type, and $F_{\omega_3}$ is not
  faithful.     
\end{exmp}
\begin{cor}\label{cor:contcoc2}
Let $\xi:Z\to Z$ be a covering map, and let $\alpha:B\to B$ be its
dual map, with range $A=C(X)$. Then a linear map
$\mathcal{L}:B\to B$ is a transfer operator for $\alpha$ if and only
if  there exists a continuous map $\omega:Z\to [0,1]$ such that 
$\sum_{z\in\pi^{-1}(x)}\omega(z)=1$, $\forall x\in X$, and   
\begin{equation}\label{eq:cortransfer}
\mathcal{L}(b)(z)=\sum_{z'\in\xi^{-1}(z)}\omega(z')b(z'),\ \forall
z\in Z. 
\end{equation}
In this case the map $\omega$ is unique. 
\end{cor}
\begin{proof}
Since $\xi$ is a covering map if and only if so is $\pi$, and
$\xi^{-1}(z)=\pi^{-1}(\pi(z))$, $\forall z\in Z$, our claims 
%immediately 
follow from Corollary~\ref{cor:contcoc}, 
Corollary~\ref{cor:transfer}, and their proofs.
\end{proof}
\subsection{Exel--Renault interaction groups}
\par Suppose again that $P$ is an Ore semigroup with enveloping group
$G$, so $G=P^{-1}P$, and
that $\theta:P\times X\to X$ is a right action, where $X$ is a  
compact Hausdorff space, and each $\theta_t$ is a covering map. Then
$\theta$ induces a left action $\alpha:P\times 
A\to A$, where $A=C(X)$ and $\alpha_t(a)=a\circ\theta_t$, $\forall
a\in A$. Suppose in addition that $V:G\to B(A)$ is an interaction
group that extends $\alpha$, that is, $V_t=\alpha_t$, $\forall t\in
P$. For each $t\in P$ let $X_t$ be the spectrum of the
$C^*$-subalgebra $A_t:=V_t(A)$, so $A_t=C(X_t)$, and let $\pi_t:X\to
X_t$ be the corresponding canonical projection. Note that 
$\pi_t(x)=\pi_t(x')\iff\theta_t(x)=\theta_t(x')$ and 
$\pi_t$ is a local homeomorphism because $\theta_t$ is. Since  
$E_t:=V_tV_{t^{-1}}:A\to A_t$ is a conditional expectation, by
Corollary~\ref{cor:contcoc} there exists a unique map
$\omega_t:X\to [0,1]$ such that
$E_t(a)(x)=\sum\limits_{\theta_t(y)=\theta_t(x)}\omega_t(y)a(y)$,
$\forall a\in A$, $x\in X$. Thus associated with the family
of conditional expectations $\{E_t\}_{t\in P}$
there is a unique map 
$\omega:P\times X\to [0,1]$ such that 
\[
E_t(a)(x)=\sum\limits_{\theta_t(y)=\theta_t(x)}\omega(t,y)a(y)
\]
$\forall t\in P, a\in A$, $x\in X$. This map $\omega$ is continuous
and satisfies 
\begin{equation}\label{eq:normal}
\sum_{y\in\theta_t^{-1}(x)}\omega(t,y)=1,
\end{equation} 
$\forall t\in P$, $x\in X$.  
The map $\omega$ also satisfies \textit{the cocycle property:}
\begin{equation}\label{eq:cocycle}
\omega(rs,x)=\omega(r,x)\omega(s,\theta_s(x)), 
\end{equation} $\forall r,s\in P$, $x\in X$,   
which reflects the fact that
$V_{s^{-1}r^{-1}}=V_{s^{-1}}V_{r^{-1}}$, $\forall r,s\in P$. Moreover,
due to the commutativity of the conditional expectations $E_s$ and
$E_r$, $\omega$~also satifies \textit{the coherence property:}        
\begin{equation}\label{eq:coherence}
\omega(s,x)W_r(C_{x,y}^{s,r})
%\sum_{u\in\theta_s^{-1}(\theta_s(x))\cap 
%     \theta_r^{-1}(\theta_r(y)) }\omega(r,u)
=\omega(r,x)W_s(C_{x,y}^{r,s}),
%\sum_{v\in\theta_r^{-1}(\theta_r(x))\cap 
%     \theta_s^{-1}(\theta_s(y)) }\omega(s,v)
\end{equation} $\forall r,s\in P$, $x,y\in X,$
where, for $S\subseteq X$, we put $W_r(S):=\sum_{x\in
  S}\omega(r,x)$, and $C_{x,y}^{s,r}=C_x^s\cap C_y^r$, with
$C_x^s=\theta_s^{-1}(\theta_s(x))$. 
\par Since $V_{r^{-1}}$ is a transfer operator for $\alpha_r$, $r\in
P$, it follows by Corollary~\ref{cor:contcoc2} that 
$V_{r^{-1}}(a)(x)= \sum_{y\in \theta_r^{-1}(x)}\omega(r,y)a(y)$,
$\forall a\in A$, $x\in X$. Then, as $V_{s^{-1}}V_s=Id_A$, 
$\forall s\in P$, if $t=r^{-1}s\in G$, $r,s\in P$, we have:
$V_t=V_{r^{-1}s}=V_{r^{-1}s}V_{s^{-1}}V_s=V_{r^{-1}}V_sV_{s^{-1}}V_s
=V_{r^{-1}}V_s=V_{r^{-1}}\alpha_s$. Therefore:     
\begin{equation}\label{eq:igext}
 V_t(a)(x)=V_{r^{-1}}\alpha_s(a)(x)
%=\sum_{y\in\theta_r^{-1}(x)}\omega(r,y)\alpha_s(a)(y)
=\sum_{y\in\theta_r^{-1}(x)}\omega(r,y)a(\theta_s(y)),
\end{equation}
$\forall a\in A$, $x\in X$. 
\par We recall from \cite{exren} the following definition:
\begin{defn}\label{defn:cocycle}
A continuous map $\omega:P\times X\to [0,1]$, such that $w(t,y)>0$
$\forall (t,y)\in P\times X$ will be called a normalized coherent
cocycle or just cocycle for the action $\theta$ if it satisfies
(\ref{eq:normal}), (\ref{eq:cocycle}) and~(\ref{eq:coherence}).   
\end{defn}     
\par It is proved in \cite[Theorem 2.8]{exren} that every
normalized coherent cocycle associated with 
  $\theta$ defines, by means of formula~\eqref{eq:igext}, an
  interaction group $V^{\omega}$ that extends  
  $\alpha$. Such a $V^{\omega}$ will be called an
  \textit{Exel--Renault interaction group}. 
 Since each $E_t$ is an index-finite type conditional expectation, as
 observed after Corollary~\ref{cor:contcoc}, we must have
 $\omega(t,x)=1/\textrm{Index} \,E_t(x)$, $\forall t\in P,x\in
 X$. Therefore, if $r,s\in P$ and $t=r^{-1}s$, from \eqref{eq:igext}
 we have: 
\begin{equation}\label{eq:igext2}
 V_t(a)(x)=V_{r^{-1}}\alpha_s(a)(x)
=\sum_{y\in\theta_r^{-1}(x)}\frac{a(\theta_s(y)) }{\textrm{Index}
  \,E_r(y)}
\end{equation}

If $V$ and $V'$ are Exel-Renault interaction groups  which extend
$\alpha$, formula~\eqref{eq:igext2} gives us a simple relation between
them: in fact, if $t=r^{-1}s\in G$, with $r,s\in P$, then with obvious
notation we have  
\begin{equation}\label{eq:unique}
V_t(a)=V_{r^{-1}}\alpha_s(a)
=V'_{r^{-1}}\Big
(\frac{\textrm{Index}E'_r}{\textrm{Index}E_r}\alpha_s(a)\Big ), 
\ \forall a\in A. 
\end{equation}    

  % Actually it is required in \cite{exren} that 
  % $\omega(r,x)\neq 0$, $\forall r\in P$, $x\in X$, but it seems to
  % us that the above-cited result still holds without such a
  % requirement.

\subsection{Dilations of Exel--Renault interaction
  groups}\label{subsec:dilexren} 
\par In this final paragraph we will show that every Exel-Renault
interaction group has a minimal faithful dilation. 
\par Let $V=V^\omega:G\to B(A)$ be an Exel-Renault interaction group
extending $\alpha:P\to B(A)$ as in the previous paragraph. Let
$(B,\beta,F)$ be the minimal admissible 
dilation of $V$. From properties 1. and~2. of
Theorem~\ref{thm:ml} we have that $B$ is the direct 
limit of copies of $A$ with 
connecting maps $\alpha_r$ (alternatively see the proof of this result
in \cite{ml}). Then $B=C(Z)$, where $(Z,\{q_r\}_{r\in
P})$ is the inverse limit of the system $(\{X_r=X\}_{r\in
P},\{\theta_{sr^{-1}}\}_{e\leq r\leq s}\}$. Concretely, we have
$Z=\{z:P\to X/\, z(r)=\theta_{sr^{-1}}(z(s)),\forall r,s\in P, 
r\leq s\}$,  and $q_r:Z\to X$ given by $q_r(z)=z(r)$. The dual (right)
action $\hat{\beta}$ of $\beta$ is described by the following  
formulae. For $t\in P$ and $z\in Z$, to compute
$\hat{\beta}_tz(r)$ we choose $s\in P$ such that $s\geq r,t$. Then
we have $\hat{\beta}_tz(r)=\theta_{sr^{-1}}(z(st^{-1}))$ 
and $\hat{\beta}_t^{-1}z(r)=z(rt)$. In other words, if $t\in P$, then 
$q_r\hat{\beta}_t=\theta_{sr^{-1}}q_{st^{-1}}$, $\forall s\geq r,t$,
and $q_r\hat{\beta}_t^{-1}=q_{rt}$. In particular
$q_e\hat{\beta}_t=\theta_tq_e$. The inclusion of $A$ into $B$ is given
by $a\mapsto aq_e$, $\forall a\in A$. Note that
$Z_x:=q_e^{-1}(x)=\{z\in Z: z(e)=x\}$ is
the inverse limit of the system $(\{Z_x(r)\}_{r\in
  P},\{\theta_{sr^{-1}}\}_{e\leq r\leq s}\}$, where
$Z_x(r)=\theta_r^{-1}(x)$.       
\begin{lem}\label{lem:basis}
Let $(\{Y_r\}_{r\in P},\{\sigma_r^s\}_{r\leq s})$ be an inverse
  limit of topological spaces with inverse limit $(Y,\{p_r\})$. Then
  the family $\mathcal{V}:=\{p_r^{-1}(V):\, r\in P,\, V\subseteq Y_r  
  \textrm{ open subset}\}$ is a basis for the topology of $Y$.  
\end{lem}
\begin{proof}
Let $V_1,\ldots,V_n$ be open subsets of $Y_r$ and $r_1,\ldots,r_n\in
P$. Suppose $y\in\bigcap_{j=1}^np_{r_j}^{-1}(V_j)$. Pick any element
$s\in P$ such that $s\geq r_j$, $\forall j=1,\ldots,n$. Then
$\sigma_{r_j}^sp_s(y)=p_{r_j}(y)$, $\forall j=1,\ldots,n$. Since
every $\sigma_{r_j}^s$ is continuous, there exists an open
neighborhood $V$ of $p_s(y)$ such that $\sigma_{r_j}^s(V)\subseteq
V_j$, $\forall j=1,\ldots,n$. Thus $y\in
p_s^{-1}(V)\subseteq\bigcap_{j=1}^np_{r_j}^{-1}(V_j)$, which shows 
that $\mathcal{V}$ is a basis for the topology of~$Y$, since 
it is already a sub--basis for it.  
\end{proof}
\par We have next the main result of this section. 
\begin{thm}\label{thm:exren}
Every Exel-Renault interaction group has a minimal faithful dilation,
unique up to isomorphism. 
\end{thm}
\begin{proof}
We will use the above just introduced notation. Let $(i,(B,\beta,F))$ 
be the minimal admissible dilation of the 
Exel-Renault interaction group $V:G\to B(A)$, with $i:A\to B$ the
natural inclusion. By \ref{prop:charcondexp} there exists a unique map
$\mu:X\to P(Z)$ such that $F(b)(x)=\int_{Z_x}b(z)d\mu_x(z)$, $\forall
b\in B$ and $x\in X$, and to see that $F$ is faithful is enough to
show that the support of $\mu_x$ is exactly $Z_x$. Observe that, since
$q_e\hat{\beta}_r=\theta_rq_e$, $\forall r\in P$, we have
$Z_x=q_e^{-1}(x)=\hat{\beta}_rq_e^{-1}(Z_x(r))=\biguplus_{y\in
Z_x(r)}\hat{\beta}_r(Z_y)$. Then if $a\in A$, $x\in X$, it follows
that $V_{r^{-1}}(a)(x)=F\beta_{r^{-1}}(a)(x)
=\int_{Z_x}aq_e(\hat{\beta}_r^{-1}(z))d\mu_x(z)$. Therefore   
\begin{equation}\label{eqn:form1}
V_{r^{-1}}(a)(x)
=\sum_{y\in Z_x(r)}
\int_{\hat{\beta}_r(Z_y)}aq_e(\hat{\beta}_r^{-1}(z))d\mu_x(z)\\ 
=\sum_{y\in Z_x(r)}\mu_x(\hat{\beta}_r(Z_y))a(y)
\end{equation}
On the other hand, if $y_0\in Z_x(r)$, then
$V_r^{-1}(a)(x)=V_r^{-1}(a)(\theta_r(y_0))
%=V_rV_{r^{-1}}(a)(y_0)
=E_r(a)(y_0)$. Thus 
\begin{equation}\label{eqn:form2}
V_r^{-1}(a)(x)
=E_r(a)(y_0)
=\sum_{y\in\theta_r^{-1}(\theta_r(y_0))}\omega(r,y)a(y)
=\sum_{y\in Z_x(r)}\omega(r,y)a(y)
\end{equation} 
Comparing (\ref{eqn:form1}) and (\ref{eqn:form2}) we see that, by the
uniqueness of $\omega(r,y)$, we must have
$\mu_x(\hat{\beta}_r(Z_y))=\omega(r,y)>0$. Since by
Lemma~\ref{lem:basis} the family 
$\{q_r^{-1}(y)=\hat{\beta}_r(Z_y): r\in P, y\in Z_x(r)\}$ is a basis
for the topology of $Z_x$, we conclude that the support of $\mu_x$ is
$Z_x$, and hence $F$ is faithful.  
\end{proof}
\begin{exmp} Consider the local homeomorphism $\theta:S^1\to S^1$
  given by $\theta(x)=x^2$, and let $\alpha:A\to A$ be its dual map,
  where $A=C(S^1)$. Then $\mathcal{L}:A\to A$ given by
  $\mathcal{L}(a)(x)=\frac{1}{2}\sum_{\{y:y^2=x\}}a(y)$ is
  a transfer operator for~$\alpha$. It is easy to see that the cocycle
  $\omega:\nt\times S^1\to [0,1]$ associated to the interaction group
  $V$ induced by $\mathcal{L}$ is given by:
  $\omega(n,y)=\frac{1}{2^n}$, $\forall n\in\nt$, $y\in S^1$. Let
  $(i,(B,\beta, F))$ be the minimal dilation of $V$. Then $B=C(Z)$,
  where the space $Z$ is the solenoid: $Z=\{z:\nt\to S^1/\
  z(n)=z(n+1)^2,\, \forall n\in\nt\}$. The inclusion $i:A\to B$ is
  the dual map of $q_0$ ( we use the notation of
  Theorem~\ref{thm:exren}: thus $q_0:Z\to S^1$ is given by
  $q_0(z)=z(0)$). The action $\beta$ is the one determined by the shift
  $\beta(b)(z)(n)=b(z(n+1))$. Thus $\hat{\beta}(z)(n)=z(n+1)$, and
  $q_0\hat{\beta}=\theta q_0$. To 
  find the corresponding conditional expectation $F:B\to A$ we need to 
  describe the measures $\mu_x$ on $Z_x=q^{-1}_0(x)=\{z\in Z:\,
  z(0)=x\}$. Note that, since every $y\in S^1$ has exactly two
  roots, then for each $n\in \nt$ the set $Z_x(n)$ ($=q_n(Z_x)$) has
  $2^n$ elements, namely the roots of $X^{2^n}-x$
  in $\C$. If $y\in Z_x(n)$, then $q_n^{-1}(y)=\hat{\beta}_n(Z_y)$, 
  and therefore $\mu_x$ is determined by the fact that we must have
  $\mu_x(q_n^{-1}(y))=\mu_x(\hat{\beta}_n(Z_y)) 
  =\omega(n,y)=\frac{1}{2^n}$.    
\par We will finish our work by explicitely computing the conditional
expectation $F$ on elements of a dense subalgebra of~$B$. Consider the 
additive group $\mathcal{M}:=\{m:\nt\to\Z/\, m(k)=0\,$%\textrm{
for all
  but a finite set of natural numbers %}
$k\}$. Given $m\in\mathcal{M}$
let $b_m\in B$ be given by $b_m=1$ if $m=0$ and
$b_m(z)=\prod_{k\in\nt}z(k)^{m(k)}$. Then 
$b_m\in B$, $b_{m_1+m_2}=b_{m_1}b_{m_2}$ and
$b_{m}^*=b_{-m}$. Therefore the Stone-Weierstrass theorem implies that
the selfadjoint subalgebra $B_0:=\gen\{b_m:m\in\mathcal{M}\}$ of $B$
is dense in $B$. For $m\in\mathcal{M}$ let $\bar{m}$ be defined as
$\bar{m}=0$, if $m=0$, and $\bar{m}=\max\{k\in\nt:\, m(k)\neq
0\}$. Then it is clear that we have
$b_m(z)=\prod_{k=0}^{\bar{m}}z(\bar{m})^{m(k)2^{\bar{m}-k}}
=z(\bar{m})^{2^{\bar{m}}\sum_{k=0}^{\bar{m}}m(k)/2^{k}}$. Note that
the previous formula also works for $m=0$. Therefore, if 
$m\in\mathcal{M}$ and $x\in S^1$:
 $ F(b_m)(x)=\int_{Z_x}b_m(z)\,d\mu_x(z)
         =\sum_{y\in
           Z_x(\bar{m})}\int_{\hat{\beta}_{\bar{m}}(Z_y)}b_m(z)\,d\mu_x(z)$  
and therefore we have 
\begin{equation}\label{eqn:F}  
F(b_m)(x)=\frac{1}{2^{\bar{m}}}\sum_{\{y:\,y^{2^{\bar{m}}}=x\}}
           y^{2^{\bar{m}}\sum_{k=0}^{\bar{m}}m(k)/2^{k}}
\end{equation}
Note that $F(b_m)\in C(S^1)$, that is, the right hand side of equation
\eqref{eqn:F} is a continuous function of $x\in S^1$. 
\end{exmp}
\bigskip
\par The main portion of the present research was done during a visit
of the author to the Universidade Federal de Santa Catarina
(Florian\'opolis, Brazil). The author wishes to thank Ruy Exel and the
people of the Departamento de Matem\'atica there for their warm
hospitatility.

\end{document}